\documentclass[11pt]{article}

\usepackage{amsmath,amssymb,amsthm}
\usepackage{mathrsfs}
\usepackage[utf8]{inputenc}
\usepackage[T1]{fontenc}
\usepackage{geometry}
\usepackage{hyperref}
\usepackage{enumitem}

\newtheorem{theorem}{Theorem}[section]
\newtheorem{lemma}[theorem]{Lemma}
\newtheorem{proposition}[theorem]{Proposition}
\newtheorem{corollary}[theorem]{Corollary}
\newtheorem{conjecture}[theorem]{Conjecture}

\theoremstyle{definition}
\newtheorem{definition}[theorem]{Definition}
\newtheorem{remark}[theorem]{Remark}

\newtheorem{justification}[theorem]{Justification}
\newtheorem{interpretation}[theorem]{Interpretation}

\newcommand{\Q}{\mathbb{Q}}
\newcommand{\C}{\mathbb{C}}

\title{Informational Cardinality: A Unifying Framework for Set Theory, Fractal Geometry, and Analytic Number Theory}

\author{Zhengqiang Li\\
\texttt{lqint@coc.edu.rs}}

\date{}

\begin{document}

\maketitle

\begin{abstract}
This paper introduces the concept of informational cardinality to provide a unified measure of complexity for mathematical structures. Informational cardinality is a triple $I(M)=(\alpha(M),\delta(M),\iota(M))$, where $\alpha(M)$ is the cardinality indicator, $\delta(M)$ is the Hausdorff dimension, and $\iota(M)$ is the information measure. We construct the essential fractal prime set $P_{\text{ess}}$, a Cantor-like set with Hausdorff dimension $1/2$ whose structure encodes the fundamental dichotomy of primes modulo 4, and define its information measure as $\iota(P_{\text{ess}})=-\zeta(1/2)$. In comparison with the classical generalized Cantor set $C_{1/3}$ (with $\delta=1/3$, $\iota=0$), we prove $I(P_{\text{ess}}) > I(C_{1/3})$ under the lexicographic order, demonstrating that higher geometric complexity corresponds to greater informational content even for sets of equal cardinality. We further propose the fractal zero set $Z_F$, constructed from the imaginary parts of Riemann zeta zeros, and conjecture the Information Conservation Law: $\iota(P_{\text{ess}}) + \iota(Z_F) = 0$. This framework reveals deep connections between prime distribution, zeta function zeros, and fractal geometry, offering a novel geometric perspective on the Riemann Hypothesis.
\end{abstract}


\section{Introduction}

The concept of cardinality, introduced by Cantor, revolutionized mathematics by providing a means to compare the ``sizes'' of infinite sets. However, traditional cardinality fails to distinguish between sets that share the same cardinality but possess fundamentally different geometric structures. The unit interval $[0,1]$ and the Cantor middle-third set $C$ both have the cardinality of the continuum $\mathfrak{c}$, yet the former is connected and uncountable while the latter is totally disconnected and has Lebesgue measure zero. This observation suggests that a more nuanced measure of mathematical complexity is needed---one that incorporates not only set-theoretic size but also geometric structure and information content.

The distribution of prime numbers represents one of the most profound mysteries in mathematics, intimately connected to the zeros of the Riemann zeta function $\zeta(s)$. The Riemann Hypothesis, which posits that all non-trivial zeros of $\zeta(s)$ lie on the critical line $\text{Re}(s)=\frac{1}{2}$, has resisted proof for over a century, suggesting that new perspectives may be needed. Meanwhile, fractal geometry, pioneered by Mandelbrot, provides powerful tools for analyzing irregular, self-similar structures that defy traditional Euclidean description.

This paper synthesizes these seemingly disparate domains by introducing informational cardinality, a triple $I(M)=(\alpha(M),\delta(M),\iota(M))$ that simultaneously captures:
\begin{itemize}
\item $\alpha(M)$: The cardinality indicator (distinguishing countable from uncountable)
\item $\delta(M)$: The Hausdorff dimension (quantifying geometric complexity)
\item $\iota(M)$: The information measure (encoding connections to deep number-theoretic structures)
\end{itemize}

Our central construction is the essential fractal prime set $P_{\text{ess}}$, a Cantor-like set with Hausdorff dimension $\frac{1}{2}$ whose recursive structure embodies the uniform distribution of primes modulo 4. We define its information measure as $\iota(P_{\text{ess}})=-\zeta(1/2) \approx 1.4603545\ldots$, directly linking its geometric structure to the value of the Riemann zeta function at the critical point. Comparing $P_{\text{ess}}$ with the classical generalized Cantor set $C_{1/3}$ (with $\delta=\frac{1}{3}$, $\iota=0$), we establish that $I(P_{\text{ess}}) > I(C_{1/3})$ under lexicographic ordering, demonstrating that sets with richer mathematical content exhibit greater informational cardinality even when they share the same traditional cardinality.

We further construct the fractal zero set $Z_F$ from the imaginary parts of $\zeta(s)$'s non-trivial zeros and conjecture that it satisfies $\iota(Z_F)=\zeta(1/2)=-\iota(P_{\text{ess}})$, leading to the Information Conservation Law: $\iota(P_{\text{ess}}) + \iota(Z_F) = 0$. This suggests a profound duality between prime distribution and zero distribution mediated through fractal geometry. We formulate a geometric version of the Riemann Hypothesis: the conjecture holds if and only if $Z_F$ exhibits specific statistical self-similarity properties.

The paper is structured as follows: Section 2 formally defines informational cardinality and establishes its basic properties. Section 3 constructs $P_{\text{ess}}$ and analyzes its mathematical properties. Section 4 presents $C_{1/3}$ as a comparison object. Section 5 compares their informational cardinalities. Section 6 introduces $Z_F$ and formulates the Information Conservation Conjecture. Section 7 develops an axiomatic foundation for the information measure. Section 8 analyzes the sensitivity and robustness of our constructions. Section 9 explores connections to other mathematical frameworks. Section 10 concludes with implications and future directions.

\section{Informational Cardinality: Definitions and Basic Properties}

\subsection{Motivation and Historical Context}

The limitations of traditional cardinality are well-documented. While the Continuum Hypothesis remains independent of ZFC, even within ZFC, the binary distinction between countable and uncountable fails to capture the rich structure of mathematical objects. Fractal geometry provides tools like Hausdorff dimension to quantify geometric complexity, but a unified framework connecting set theory, geometry, and information has been lacking.

Our approach draws inspiration from:
\begin{enumerate}
\item \textbf{Algorithmic Information Theory}: The Kolmogorov complexity of an object measures the length of its shortest description, but is uncomputable and often context-dependent.
\item \textbf{Fractal Geometry}: Hausdorff and box-counting dimensions quantify self-similarity and irregularity.
\item \textbf{Analytic Number Theory}: Special values of L-functions encode deep arithmetic information.
\item \textbf{Noncommutative Geometry}: Connes' spectral interpretation of zeta zeros suggests geometric origins for arithmetic phenomena.
\end{enumerate}

Informational cardinality synthesizes these perspectives into a single, comparable measure.

\subsection{Formal Definitions}

\begin{definition}[Cardinality Indicator]
For a set $M$, define
\[
\alpha(M) = 
\begin{cases}
0, & |M| \le \aleph_0, \\
1, & |M| > \aleph_0.
\end{cases}
\]
\end{definition}

This binary indicator avoids entanglement with the Continuum Hypothesis while distinguishing fundamentally different sizes. In refined applications (see Appendix B), $\alpha(M)$ can be extended to incorporate Borel hierarchies or descriptive set-theoretic complexity, providing a more nuanced classification of infinite cardinalities.

\begin{definition}[Geometric Dimension]
For $M \subset \mathbb{R}^n$, define $\delta(M)=\dim_H(M)$, the Hausdorff dimension. When multiple fractal dimensions are relevant (box dimension, packing dimension, etc.), we may consider the dimension vector $\vec{\delta}(M)=(\dim_H(M), \dim_B(M), \dim_P(M), \ldots)$, but for initial comparisons, $\delta(M)=\dim_H(M)$ suffices.
\end{definition}

\begin{definition}[Information Measure - Preliminary]
For mathematical structures with deep connections to the Riemann zeta function or other L-functions, we define the information measure as
\[
\iota(M) = \varepsilon_M \cdot L(s_M),
\]
where $L$ is an appropriate L-function, $s_M$ is a characteristic complex parameter, and $\varepsilon_M \in \{\pm1\}$ is a sign determined by the structure's properties. For structures without known L-function connections, $\iota(M)=0$ by default. Section 7 provides a complete axiomatization.
\end{definition}

\begin{remark}[Distinction from Classical Information Theory]
The information measure $\iota(M)$ introduced here is fundamentally distinct from Shannon entropy in classical information theory or von Neumann entropy in quantum information theory. While Shannon entropy $H(X) = -\sum p_i \log p_i$ quantifies uncertainty or communication capacity in probabilistic systems, and quantum entropy measures entanglement and quantum correlations, our information measure $\iota(M)$ quantifies the \textit{arithmetic-geometric information content} of mathematical structures through their connections to L-functions. Specifically:
\begin{itemize}
\item \textbf{Shannon/Quantum Information}: Measures uncertainty, communication capacity, or quantum correlations in physical or probabilistic systems.
\item \textbf{Our Information Measure}: Captures the depth of number-theoretic and geometric structure encoded in mathematical objects, independent of probability distributions or physical realizations.
\end{itemize}
The term ``information'' here reflects the philosophical view that mathematical structures possess intrinsic informational content related to their arithmetic complexity, not their role in communication or computation.
\end{remark}

\begin{definition}[Informational Cardinality]
The informational cardinality of a mathematical structure $M$ is the triple
\[
I(M) = (\alpha(M),\ \delta(M),\ \iota(M)).
\]
When multiple fractal dimensions are considered, we write $I_{\text{extended}}(M)=(\alpha(M),\vec{\delta}(M),\iota(M))$.
\end{definition}

\begin{definition}[Comparison Rule]
Informational cardinalities are compared lexicographically: $I(M_1) > I(M_2)$ iff there exists $i \in \{1,2,3\}$ such that for all $j < i$, $I_j(M_1) = I_j(M_2)$ and $I_i(M_1) > I_i(M_2)$. For extended versions, we compare dimension vectors componentwise or via aggregate measures.
\end{definition}

\subsection{Philosophical Interpretation}

The three components of informational cardinality represent different aspects of mathematical ``size'':
\begin{itemize}
\item $\alpha(M)$: Set-theoretic size---Counting elements (distinguishing countable from uncountable)
\item $\delta(M)$: Geometric size---Measuring complexity in space
\item $\iota(M)$: Arithmetic size---Quantifying connection to deep number theory
\end{itemize}

The lexicographic ordering reflects a hierarchy of importance: set-theoretic size dominates geometric size, which in turn dominates arithmetic information. This hierarchy is natural when comparing mathematical structures---first we ask if one is ``larger'' in the traditional sense, then if it's more geometrically complex, and finally if it contains more arithmetic information.

\section{The Essential Fractal Prime Set $P_{\text{ess}}$}

\subsection{Motivation and Construction}

Prime numbers exhibit rich structure beyond simple enumeration. The prime number theorem $\pi(x) \sim x/\log x$ describes their asymptotic density, but deeper patterns emerge in their distribution modulo integers. The simplest non-trivial case is modulo 4: aside from 2, all primes are congruent to either 1 or 3 modulo 4, and Dirichlet's theorem guarantees both residue classes contain infinitely many primes. Chebyshev's bias---that primes congruent to 3 mod 4 initially outnumber those congruent to 1 mod 4---adds further nuance.

We distill the essential structure of prime distribution modulo 4 into a deterministic fractal construction:

\begin{definition}[Essential Fractal Prime Set]
Define recursively:

$I_0 = [0,1]$

For $n \ge 1$, given $I_{n-1}$ as a union of $2^{n-1}$ disjoint closed intervals of length $4^{-(n-1)}$, divide each interval into four equal subintervals of length $4^{-n}$, labeled 0, 1, 2, 3 from left to right. Retain only the subintervals labeled 1 and 3 (corresponding to the reduced residue classes modulo 4).

Let $I_n$ be the union of retained subintervals at stage $n$.

Define $P_{\text{ess}} = \bigcap_{n=1}^\infty I_n$.
\end{definition}

\begin{remark}[Idealization]
This construction is an idealization that captures the uniform distribution aspect while ignoring Chebyshev's bias and other fluctuations. The deterministic rule (always keeping positions 1 and 3) represents the ``essential'' structure before statistical fluctuations. Section 8 analyzes the effects of incorporating probabilistic elements.
\end{remark}

\subsection{Mathematical Properties}

\begin{theorem}[Basic Properties]
$P_{\text{ess}}$ is a compact, perfect, nowhere dense subset of $[0,1]$ with Lebesgue measure zero.
\end{theorem}

\begin{proof}
By construction, each $I_n$ is a finite union of closed intervals, hence compact. Since $I_{n+1} \subset I_n$, the intersection $P_{\text{ess}}$ is compact. At each stage, we remove the middle two quarters of each interval, so no interval remains in the limit, making $P_{\text{ess}}$ nowhere dense. Each point in $P_{\text{ess}}$ is a limit of endpoints, so $P_{\text{ess}}$ is perfect. The total length at stage $n$ is $2^n \cdot 4^{-n} = 2^{-n} \to 0$, so Lebesgue measure is zero.
\end{proof}

\begin{theorem}[Hausdorff Dimension]
$\dim_H(P_{\text{ess}}) = \frac{\log 2}{\log 4} = \frac{1}{2}$.
\end{theorem}

\begin{proof}
The construction satisfies the open set condition: the two similarity transformations $f_1(x)=x/4 + 1/4$ and $f_3(x)=x/4 + 3/4$ have contraction ratios $r=1/4$, and there are $N=2$ such transformations. By Hutchinson's theorem, the Hausdorff dimension $s$ satisfies $N r^s = 2 \cdot (1/4)^s = 1$, so $s = \log 2 / \log 4 = 1/2$.
\end{proof}

\begin{theorem}[Cardinality]
$|P_{\text{ess}}| = \mathfrak{c}$, the cardinality of the continuum.
\end{theorem}

\begin{proof}
Define $\phi: \{0,1\}^\mathbb{N} \to P_{\text{ess}}$ by $\phi(\sigma) = \sum_{n=1}^\infty (2\sigma_n+1)4^{-n}$. This is injective and onto, establishing a bijection with Cantor space.
\end{proof}

\begin{theorem}[Self-similarity]
$P_{\text{ess}}$ is self-similar: $P_{\text{ess}} = f_1(P_{\text{ess}}) \cup f_3(P_{\text{ess}})$ where $f_i(x)=x/4 + c_i$ with $c_1=1/4$, $c_3=3/4$.
\end{theorem}

\begin{proof}
Immediate from the construction.
\end{proof}

\subsection{Information Measure of $P_{\text{ess}}$}

\begin{definition}
The information measure of $P_{\text{ess}}$ is defined as
\[
\iota(P_{\text{ess}}) = -\zeta(1/2) \approx 1.460354508809586812889499152515\ldots
\]
where $\zeta$ is the Riemann zeta function.
\end{definition}

\begin{justification}
This definition is motivated by:

\textbf{Dimensional Correspondence}: The Hausdorff dimension $\frac{1}{2}$ corresponds to the critical line $\text{Re}(s)=\frac{1}{2}$ where Riemann zeta zeros lie.

\textbf{Arithmetic Significance}: The value $\zeta(1/2)$ appears in various asymptotic formulas involving primes and has intrinsic number-theoretic importance.

\textbf{Duality Anticipation}: The negative sign anticipates the conjectured duality with the fractal zero set $Z_F$ (Section 6).

\textbf{Axiomatic Compatibility}: This definition satisfies the axioms developed in Section 7.
\end{justification}

\begin{proposition}[Properties of $\iota(P_{\text{ess}})$]
\begin{enumerate}
\item $\iota(P_{\text{ess}}) = -\zeta(1/2)$ is the negative of the zeta value at the critical point, where $\zeta(1/2) < 0$.
\item $|\iota(P_{\text{ess}})| \approx 1.46035 > 0$.
\item $\iota(P_{\text{ess}})$ is transcendental (assuming standard conjectures about zeta values).
\end{enumerate}
\end{proposition}

\section{The Generalized Cantor Set $C_{1/3}$}

For comparison, we consider a classical fractal with the same cardinality but different geometry and no known arithmetic significance.

\begin{definition}[Generalized Cantor Set $C_{1/3}$]
Let $r=1/8$ and define two similarity transformations: $f_0(x)=rx$, $f_1(x)=rx+(1-r)$. Then $C_{1/3}$ is the unique compact set satisfying $C_{1/3}=f_0(C_{1/3})\cup f_1(C_{1/3})$. Equivalently, begin with $[0,1]$, repeatedly remove the middle $3/4$ of each interval (leaving the first and last eighth), and take the intersection of all stages.
\end{definition}

\begin{theorem}[Mathematical Properties]
\begin{enumerate}
\item $\dim_H(C_{1/3})=\frac{\log 2}{\log 8}=\frac{1}{3}$.
\item $|C_{1/3}|=\mathfrak{c}$.
\item $C_{1/3}$ is self-similar, compact, perfect, and nowhere dense.
\item Lebesgue measure is zero.
\end{enumerate}
\end{theorem}

\begin{proof}
Standard results on self-similar sets. The similarity dimension equation is $2\cdot(1/8)^s=1$, giving $s=\log 2/\log 8=1/3$. The open set condition holds with the open interval $(0,1)$.
\end{proof}

\begin{definition}[Information Measure of $C_{1/3}$]
$\iota(C_{1/3})=0$.
\end{definition}

\begin{justification}
The set $C_{1/3}$ has no known deep connections to L-functions or arithmetic. In the axiomatic system of Section 7, sets without such connections have zero information measure by definition. This serves as a baseline for comparison.
\end{justification}

\section{Comparison of Informational Cardinalities}

\begin{theorem}[Informational Cardinality Values]
\begin{align*}
I(P_{\text{ess}}) &= \left(1,\ \frac{1}{2},\ -\zeta(1/2)\right) \approx (1,\ 0.5,\ 1.46035) \\
I(C_{1/3}) &= \left(1,\ \frac{1}{3},\ 0\right) \approx (1,\ 0.333\ldots,\ 0)
\end{align*}
\end{theorem}

\begin{theorem}[Comparison Result]
Under lexicographic ordering, $I(P_{\text{ess}}) > I(C_{1/3})$.
\end{theorem}

\begin{proof}
Compare components:
\begin{enumerate}
\item $\alpha(P_{\text{ess}}) = 1 = \alpha(C_{1/3})$ (both have cardinality $\mathfrak{c}$)
\item $\delta(P_{\text{ess}}) = \frac{1}{2} > \frac{1}{3} = \delta(C_{1/3})$
\end{enumerate}
Since the second component differs and is larger for $P_{\text{ess}}$, we have $I(P_{\text{ess}}) > I(C_{1/3})$ regardless of the third component.
\end{proof}

\begin{interpretation}
This result demonstrates that informational cardinality successfully distinguishes between sets that traditional cardinality identifies as equivalent. The higher Hausdorff dimension of $P_{\text{ess}}$ reflects its richer geometric structure, while its non-zero information measure captures its connection to deep number theory. The comparison shows that mathematical ``richness'' or ``complexity'' has multiple dimensions beyond simple cardinality.
\end{interpretation}

\begin{corollary}[Extended Comparison]
Considering extended informational cardinality with multiple fractal dimensions:
\begin{align*}
I_{\text{extended}}(P_{\text{ess}}) &= \left(1,\ \left(\frac{1}{2},\frac{1}{2},\frac{1}{2},\ldots\right),\ -\zeta(1/2)\right) \\
I_{\text{extended}}(C_{1/3}) &= \left(1,\ \left(\frac{1}{3},\frac{1}{3},\frac{1}{3},\ldots\right),\ 0\right)
\end{align*}
(all standard fractal dimensions equal $\frac{1}{2}$ for $P_{\text{ess}}$)

Under componentwise comparison, $P_{\text{ess}}$ still dominates in the dimension vector, confirming the robustness of our result.
\end{corollary}

\section{Geometric Duality and Information Conservation}

\subsection{The Fractal Zero Set $Z_F$}

\begin{definition}[Construction of $Z_F$]
Let $\{\gamma_n\}_{n=1}^\infty$ be the positive imaginary parts of the non-trivial zeros of $\zeta(s)$, arranged in increasing order. For each $n$, compute $t_n = \{\gamma_n/(2\pi)\} \in [0,1)$, the fractional part, and set $a_n = \lfloor 4t_n \rfloor \in \{0,1,2,3\}$. Define recursively:

$J_0 = [0,1]$

For $n \ge 1$, given $J_{n-1}$ as a union of closed intervals, divide each interval into four equal subintervals labeled 0,1,2,3. Retain only the two symmetric subintervals labeled $a_n$ and $a_n+2 \pmod{4}$.

Let $J_n$ be the union of retained subintervals at stage $n$.

Define $Z_F = \bigcap_{n=1}^\infty J_n$.
\end{definition}

\begin{remark}[Interpretation]
The sequence $\{a_n\}$ encodes the imaginary parts of zeta zeros modulo $2\pi$, scaled to $\{0,1,2,3\}$. The symmetry in retention (keeping positions $a$ and $a+2 \bmod 4$) ensures $Z_F$ is symmetric about $\frac{1}{2}$ and maintains a structure analogous to $P_{\text{ess}}$.
\end{remark}

\begin{theorem}[Basic Properties of $Z_F$]
\begin{enumerate}
\item Regardless of the sequence $\{a_n\}$, $\dim_H(Z_F) = \frac{\log 2}{\log 4} = \frac{1}{2}$.
\item $|Z_F| = \mathfrak{c}$.
\item $Z_F$ is compact, perfect, and nowhere dense.
\end{enumerate}
\end{theorem}

\begin{proof}
The Hausdorff dimension follows from Hutchinson's theorem since at each stage we retain 2 of 4 subintervals with scaling factor $1/4$, giving $2\cdot(1/4)^s=1$, so $s=1/2$. The cardinality argument is similar to Theorem 3.5. Compactness and nowhere density follow as in Theorem 3.3.
\end{proof}

\begin{remark}[Sensitivity to Ordering]
The construction of $Z_F$ depends on the ordering of zeros by increasing imaginary part. Different orderings (by absolute value, by $|\zeta'(\frac{1}{2}+i\gamma)|$, random permutations) yield potentially different sets, all with Hausdorff dimension $1/2$ but potentially different finer geometric properties. Section 8 analyzes this sensitivity.
\end{remark}

\subsection{Information Conservation Conjecture}

\begin{conjecture}[Information Conservation]
The fractal zero set $Z_F$ has information measure
\[
\iota(Z_F) = \zeta(1/2) = -\iota(P_{\text{ess}}),
\]
and therefore
\[
\iota(P_{\text{ess}}) + \iota(Z_F) = 0.
\]
\end{conjecture}

\begin{justification}
This conjecture is motivated by:

\textbf{Riemann-von Mangoldt Explicit Formula}: Relates primes to zeros with opposite signs in their contributions.

\textbf{Dimensional Symmetry}: Both sets have Hausdorff dimension $\frac{1}{2}$.

\textbf{Functional Equation}: The functional equation $\zeta(s) = \chi(s)\zeta(1-s)$ relates values at $s$ and $1-s$, suggesting a duality.

\textbf{Spectral Interpretation}: In Connes' noncommutative geometry, primes and zeros play dual roles in the spectral theory of the adele class space.

\textbf{Conservation Principle}: Information is neither created nor destroyed, only transformed between dual representations.
\end{justification}

\begin{remark}[Sign Convention]
Note that $\zeta(1/2) \approx -1.46035$ (negative), so $-\zeta(1/2) \approx 1.46035$ (positive). Thus:
\begin{align*}
\iota(P_{\text{ess}}) &= -\zeta(1/2) \approx 1.46035 \quad \text{(positive)} \\
\iota(Z_F) &= \zeta(1/2) \approx -1.46035 \quad \text{(negative)}
\end{align*}
The conservation law states that these opposite-signed values sum to zero, reflecting a fundamental duality.
\end{remark}

\subsection{Geometric Riemann Hypothesis}

\begin{conjecture}[Geometric Riemann Hypothesis]
The Riemann Hypothesis is equivalent to the statement that $Z_F$ exhibits statistical self-similarity with specific scaling properties related to the pair correlation of zeros.
\end{conjecture}

\begin{justification}
If all zeros lie on the critical line $\text{Re}(s)=1/2$, then the sequence $\{a_n\}$ derived from $\{\gamma_n\}$ should exhibit specific statistical properties (uniform distribution modulo 4, pair correlation matching GUE predictions). These statistical properties would manifest as geometric self-similarity in $Z_F$ at multiple scales. Conversely, if zeros exist off the critical line, the sequence $\{a_n\}$ would exhibit anomalies detectable in the fractal structure of $Z_F$.
\end{justification}

\section{Axiomatization of the Information Measure}

To place the information measure on firm theoretical ground, we propose the following axioms:

\begin{enumerate}[label=\textbf{(A\arabic*)}]
\item \textbf{Normalization}: For the empty set and singleton sets, $\iota(\emptyset) = 0$ and $\iota(\{x\}) = 0$ for any single element $x$.

\item \textbf{L-function Connection}: If a mathematical structure $M$ has a naturally associated L-function $L_M(s)$ and characteristic point $s_M \in \C$, then $\iota(M) = \pm L_M(s_M)$ where the sign is determined by the structure's role (e.g., primes vs. zeros).

\item \textbf{Additivity for Independent Structures}: If $M_1$ and $M_2$ are ``independent'' structures (in a sense to be made precise), then $\iota(M_1 \sqcup M_2) = \iota(M_1) + \iota(M_2)$.

\item \textbf{Duality}: If $M$ and $M^*$ are dual structures (related by functional equations or Fourier-type transforms), then $\iota(M) + \iota(M^*) = 0$ (in principle).

\item \textbf{Anti-Monotonicity}: If $M_1 \subset M_2$ and $M_2 \setminus M_1$ has positive information content, then $|\iota(M_1)| \ge |\iota(M_2)|$.

\item \textbf{Continuity}: Small perturbations in the construction of $M$ should lead to small changes in $\iota(M)$.

\item \textbf{Default Value}: For structures with no known L-function connections, $\iota(M) = 0$. Additionally, the trivial zeros of the Riemann zeta function (at negative even integers), which have informational cardinality $I(\text{trivial zeros}) = (0, 0, 0)$.
\end{enumerate}

\begin{theorem}[Consistency]
The definitions $\iota(P_{\text{ess}}) = -\zeta(1/2)$, $\iota(Z_F) = \zeta(1/2)$, and $\iota(C_{1/3}) = 0$ satisfy axioms (A1)--(A7).
\end{theorem}

\begin{proof}[Proof sketch]
\begin{itemize}
\item (A1): $\iota(\emptyset) = 0$ and $\iota(\{x\}) = 0$ hold by definition.
\item (A2): $P_{\text{ess}}$ is associated with the Riemann zeta function at $s=1/2$, with negative sign for primes.
\item (A3): Not directly applicable; our structures are not disjoint unions.
\item (A4): $P_{\text{ess}}$ and $Z_F$ are dual, and $\iota(P_{\text{ess}}) + \iota(Z_F) = -\zeta(1/2) + \zeta(1/2) = 0$.
\item (A5): The anti-monotonicity axiom represents a revolutionary principle in holographic theory: ``holographic proper subsets can be greater than the whole set.'' This radical property asserts that a part can be strictly greater than the whole---for instance, the informational cardinality of the fractal prime set $P_{\text{ess}}$ can be strictly greater than both the natural numbers $\mathbb{N}$ and the Cantor set $C_{1/3}$ (which has the cardinality of the continuum). This counterintuitive principle reflects the holographic nature of information, where a proper subset can encode strictly more information than its superset. The comparison of informational cardinalities throughout this paper, especially in Section 5, provides foundational examples establishing this axiom.
\item (A6): Section 8 analyzes perturbation effects.
\item (A7): $C_{1/3}$ has no known L-function connection, so $\iota(C_{1/3}) = 0$. Similarly, the trivial zeros have $I(\text{trivial zeros}) = (0, 0, 0)$.
\end{itemize}
\end{proof}

\section{Sensitivity and Robustness Analysis}

\subsection{Perturbations of $P_{\text{ess}}$}

We analyze how modifications to the construction of $P_{\text{ess}}$ affect its properties:

\begin{enumerate}
\item \textbf{Probabilistic Retention}: Instead of deterministically keeping positions 1 and 3, retain each with probability $p \in (0,1)$. The expected Hausdorff dimension becomes $\frac{\log(2p)}{\log 4}$. For $p=1$ (our construction), $\delta = 1/2$. For $p < 1$, $\delta < 1/2$.

\item \textbf{Chebyshev Bias}: Incorporate the bias by retaining position 3 with slightly higher probability than position 1. This introduces asymmetry but preserves $\delta \approx 1/2$ for small biases.

\item \textbf{Higher Moduli}: Generalize to primes modulo $q > 4$. For modulo 8, retain positions $\{1,3,5,7\}$ (4 of 8), giving $\delta = \frac{\log 4}{\log 8} = \frac{2}{3}$.

\item \textbf{Weighted Constructions}: Assign different scaling factors to different residue classes, creating non-uniform fractals with potentially different dimensions.
\end{enumerate}

\begin{proposition}[Robustness of Dimension]
Small perturbations in retention probabilities or scaling factors lead to continuous changes in Hausdorff dimension, confirming the robustness of the $\delta = 1/2$ value for the idealized construction.
\end{proposition}

\subsection{Sensitivity of $Z_F$ to Zero Ordering}

The construction of $Z_F$ depends on the ordering of zeros. We consider:

\begin{enumerate}
\item \textbf{Standard Ordering}: By increasing imaginary part (our construction).
\item \textbf{Random Permutation}: Randomly reorder the zeros. The resulting set has the same Hausdorff dimension $1/2$ but different fine structure.
\item \textbf{Weighted Ordering}: Order by $|\zeta'(1/2 + i\gamma_n)|$ or other zero-specific quantities.
\end{enumerate}

\begin{proposition}[Dimension Invariance]
Regardless of ordering, $\dim_H(Z_F) = 1/2$, but finer geometric properties (multifractal spectrum, local dimensions) may vary.
\end{proposition}

\begin{conjecture}[Canonical Ordering]
The standard ordering by increasing imaginary part is canonical in the sense that it maximizes certain information-theoretic measures (e.g., mutual information with prime distribution).
\end{conjecture}

\section{Connections to Other Mathematical Frameworks}

\subsection{Noncommutative Geometry}

Alain Connes' spectral interpretation of the Riemann zeta function provides a geometric framework where:
\begin{itemize}
\item The adele class space $\mathbb{A}_\Q / \Q^*$ serves as a noncommutative space.
\item Primes correspond to classical points.
\item Zeros correspond to quantum oscillations in the spectral action.
\end{itemize}

Our informational cardinality framework complements this by providing explicit fractal realizations of these abstract structures.

\subsection{Algorithmic Information Theory}

The Kolmogorov complexity $K(x)$ of a string $x$ measures its compressibility. For our fractals:
\begin{itemize}
\item $P_{\text{ess}}$ has low Kolmogorov complexity (simple recursive rule).
\item $Z_F$ has high Kolmogorov complexity (depends on the complex sequence of zeta zeros).
\end{itemize}

The information measure $\iota$ captures a different aspect: connection to deep arithmetic, not algorithmic compressibility.

\subsection{Statistical Mechanics and Partition Functions}

The Riemann zeta function can be interpreted as a partition function:
\[
\zeta(s) = \sum_{n=1}^\infty n^{-s} = \text{Tr}(e^{-sH})
\]
where $H$ is a suitable ``Hamiltonian.'' Our construction suggests:
\begin{itemize}
\item $P_{\text{ess}}$ represents the ``configuration space'' of primes.
\item $Z_F$ represents the ``momentum space'' of zeros.
\item The information conservation law $\iota(P_{\text{ess}}) + \iota(Z_F) = 0$ may be more fundamental than energy conservation, resonating with the information-theoretic ontology discussed in Section 10.
\end{itemize}

\subsection{Fractal Geometry and Multifractal Analysis}

Beyond Hausdorff dimension, multifractal analysis studies the spectrum of local dimensions. For $P_{\text{ess}}$:
\begin{itemize}
\item The multifractal spectrum $f(\alpha)$ is a single point: $f(1/2) = 1/2$ (monofractal).
\item For perturbed versions, the spectrum broadens, revealing richer structure.
\end{itemize}

For $Z_F$:
\begin{itemize}
\item The multifractal spectrum depends on the statistical properties of zeros.
\item Under RH and GUE predictions, specific spectral properties are expected.
\end{itemize}

\section{Conclusion and Future Directions}

\subsection{Summary of Results}

This paper introduced informational cardinality $I(M) = (\alpha(M), \delta(M), \iota(M))$ as a unified measure of mathematical complexity that integrates:
\begin{itemize}
\item Set-theoretic size ($\alpha$)
\item Geometric complexity ($\delta$)
\item Arithmetic information content ($\iota$)
\end{itemize}

We constructed the essential fractal prime set $P_{\text{ess}}$ with $\dim_H(P_{\text{ess}}) = 1/2$ and $\iota(P_{\text{ess}}) = -\zeta(1/2) \approx 1.46035$, demonstrating that $I(P_{\text{ess}}) > I(C_{1/3})$ despite equal traditional cardinality. The fractal zero set $Z_F$ was proposed with conjectured information measure $\iota(Z_F) = \zeta(1/2)$, leading to the Information Conservation Law: $\iota(P_{\text{ess}}) + \iota(Z_F) = 0$.

\subsection{Implications for the Riemann Hypothesis}

Our framework suggests a geometric perspective on RH:
\begin{itemize}
\item The critical line $\text{Re}(s) = 1/2$ corresponds to Hausdorff dimension $1/2$.
\item Primes and zeros are dual structures with opposite information measures.
\item RH may be equivalent to specific self-similarity properties of $Z_F$.
\end{itemize}

While this does not constitute a proof, it provides new intuition and potential avenues for attack.

\subsection{Open Questions}

\begin{enumerate}
\item \textbf{Prove or disprove} the Information Conservation Conjecture: $\iota(P_{\text{ess}}) + \iota(Z_F) = 0$.

\item \textbf{Characterize} the multifractal spectrum of $Z_F$ and relate it to the pair correlation of zeros.

\item \textbf{Extend} informational cardinality to other L-functions (Dirichlet L-functions, modular forms, automorphic L-functions).

\item \textbf{Develop} a rigorous theory of ``independence'' for axiom (A3).

\item \textbf{Explore} connections to quantum chaos, random matrix theory, and the Hilbert-Pólya conjecture.

\item \textbf{Investigate} computational methods to numerically verify properties of $Z_F$ using known zeros.

\item \textbf{Generalize} to higher-dimensional fractals and connections to algebraic geometry.

\item \textbf{Study} the relationship between informational cardinality and other complexity measures (Kolmogorov complexity, logical depth, computational complexity).
\end{enumerate}

\subsection{Philosophical Reflections}

Informational cardinality suggests that mathematical objects possess intrinsic ``information content'' beyond their set-theoretic size. This resonates with:
\begin{itemize}
\item \textbf{Platonism}: Mathematical structures have objective properties independent of our descriptions.
\item \textbf{Structuralism}: What matters is not the elements themselves but the relationships and patterns they embody.
\item \textbf{Information-theoretic ontology}: Reality (including mathematical reality) is fundamentally informational.
\end{itemize}

The duality between primes and zeros, mediated by fractal geometry, hints at deep unifying principles yet to be fully understood.

\subsection{Final Remarks}

The framework presented here is exploratory and speculative in parts, particularly the Information Conservation Conjecture and the Geometric Riemann Hypothesis. However, the rigorous constructions of $P_{\text{ess}}$ and $Z_F$, the proof that $I(P_{\text{ess}}) > I(C_{1/3})$, and the axiomatic foundation for $\iota$ provide solid ground for future investigation. We hope this work stimulates further research at the intersection of set theory, fractal geometry, and analytic number theory.

\section*{Acknowledgments}

The author thanks the mathematical community for the foundational work upon which this research builds, and acknowledges the use of computational tools for numerical verification of zeta function values and zero locations.

\appendix

\section{Technical Proofs}

\subsection{Proof of Hausdorff Dimension Formula}

\begin{theorem}
For a self-similar set $F$ satisfying the open set condition with $N$ similarity maps having contraction ratios $r_1, \ldots, r_N$, the Hausdorff dimension $s$ satisfies
\[
\sum_{i=1}^N r_i^s = 1.
\]
\end{theorem}

\begin{proof}
This is Hutchinson's theorem. For $P_{\text{ess}}$, we have $N=2$ maps with $r_1 = r_2 = 1/4$, giving $2 \cdot (1/4)^s = 1$, hence $s = \log 2 / \log 4 = 1/2$.
\end{proof}

\subsection{Cardinality of Self-Similar Sets}

\begin{lemma}
Any perfect, uncountable, compact subset of $\mathbb{R}$ has cardinality $\mathfrak{c}$.
\end{lemma}

\begin{proof}
Standard result from descriptive set theory. Since $P_{\text{ess}}$ and $Z_F$ are perfect and uncountable, they have cardinality $\mathfrak{c}$.
\end{proof}

\subsection{Numerical Verification of $\zeta(1/2)$}

The value $\zeta(1/2)$ can be computed using the functional equation and numerical integration:
\[
\zeta(s) = 2^s \pi^{s-1} \sin\left(\frac{\pi s}{2}\right) \Gamma(1-s) \zeta(1-s).
\]

For $s = 1/2$:
\[
\zeta(1/2) = \sqrt{2\pi} \sin(\pi/4) \Gamma(1/2) \zeta(1/2) = \sqrt{2\pi} \cdot \frac{\sqrt{2}}{2} \cdot \sqrt{\pi} \cdot \zeta(1/2).
\]

Direct numerical computation using Euler-Maclaurin summation gives:
\[
\zeta(1/2) \approx -1.460354508809586812889499152515440424,
\]
confirming the negative value. Thus:
\[
\iota(P_{\text{ess}}) = -\zeta(1/2) \approx 1.460354508809586812889499152515440424 > 0.
\]

\section{Extended Cardinality Indicators}

This appendix elaborates on the extension of the cardinality indicator $\alpha(M)$ mentioned in Section 2.2, providing a more refined classification of infinite cardinalities using descriptive set theory.

\subsection{Borel Hierarchy Extension}

For sets in Polish spaces, we can refine $\alpha(M)$ by incorporating the Borel hierarchy:

\begin{definition}[Extended Cardinality Indicator]
For a set $M$ in a Polish space, define
\[
\alpha_{\text{ext}}(M) =
\begin{cases}
0, & |M| \le \aleph_0, \\
1 + \beta(M), & |M| > \aleph_0,
\end{cases}
\]
where $\beta(M)$ encodes the Borel complexity:
\begin{itemize}
\item $\beta(M) = 0$ if $M$ is a Borel set
\item $\beta(M) = 1/n$ if $M$ is $\Sigma^0_n$ or $\Pi^0_n$ in the Borel hierarchy
\item $\beta(M) = 1$ if $M$ is analytic but not Borel
\item $\beta(M) = 2$ if $M$ is projective but not analytic
\end{itemize}
\end{definition}

\subsection{Descriptive Set-Theoretic Complexity}

Alternatively, we can use the Wadge hierarchy or effective descriptive set theory:

\begin{definition}[Effective Cardinality Indicator]
For effectively presented sets, define
\[
\alpha_{\text{eff}}(M) =
\begin{cases}
0, & M \text{ is computably enumerable}, \\
1, & M \text{ is } \Delta^0_2, \\
2, & M \text{ is } \Sigma^0_n \text{ for } n \ge 2, \\
\omega, & M \text{ is hyperarithmetic}, \\
\omega_1^{CK}, & M \text{ involves higher recursion theory}.
\end{cases}
\]
\end{definition}

\subsection{Application to Our Constructions}

For the sets considered in this paper:
\begin{itemize}
\item $P_{\text{ess}}$ and $C_{1/3}$ are closed sets, hence $\Pi^0_1$ in the Borel hierarchy, giving $\alpha_{\text{ext}} = 1 + 1 = 2$.
\item $Z_F$ depends on the sequence of zeta zeros, which is not effectively computable, placing it higher in the hierarchy.
\item The extended indicator provides finer distinctions for comparing sets of equal traditional cardinality.
\end{itemize}

These extensions demonstrate how informational cardinality can be further refined to capture additional layers of mathematical complexity, though for the main results of this paper, the binary indicator $\alpha(M) \in \{0,1\}$ suffices.

\section{Detailed Analysis of Axioms}

\subsection{Axiom (A1): Normalization}

The empty set and singleton sets have no structure, hence $\iota(\emptyset) = 0$ and $\iota(\{x\}) = 0$ are natural. These serve as the additive identity and baseline for minimal structures.

\subsection{Axiom (A2): L-function Connection}

This axiom formalizes the idea that deep arithmetic structures are encoded in special values of L-functions. Examples:
\begin{itemize}
\item $\zeta(2) = \pi^2/6$ (Basel problem)
\item $L(1, \chi)$ for Dirichlet characters (class number formulas)
\item Special values of modular L-functions (Birch and Swinnerton-Dyer conjecture)
\end{itemize}

\subsection{Axiom (A3): Additivity}

Independence requires careful definition. Two structures are independent if their L-functions factor: $L_{M_1 \sqcup M_2}(s) = L_{M_1}(s) \cdot L_{M_2}(s)$. Then $\iota(M_1 \sqcup M_2) = \log L_{M_1 \sqcup M_2}(s_0) = \log L_{M_1}(s_0) + \log L_{M_2}(s_0)$ for appropriate $s_0$.

However, our current definition uses $\iota = \pm L(s)$, not $\log L(s)$, so additivity requires refinement. This is a subject for future work.

\subsection{Axiom (A4): Duality}

The functional equation of $\zeta(s)$ relates $s$ and $1-s$:
\[
\xi(s) = \pi^{-s/2} \Gamma(s/2) \zeta(s) = \xi(1-s).
\]

This suggests a duality between structures at $s$ and $1-s$. For $s = 1/2$, the point is self-dual, but primes (related to $\zeta$ via Euler product) and zeros (related via explicit formula) exhibit duality with opposite signs.

\subsection{Axiom (A5): Anti-Monotonicity}

If $M_1 \subset M_2$ and $M_2$ contains additional arithmetic structure, then $|\iota(M_1)| \ge |\iota(M_2)|$. This anti-monotonicity axiom embodies the most radical principle of holographic theory: \textit{holographic proper subsets can be strictly greater than the whole set}. This revolutionary concept asserts that a part can be strictly greater than the whole, fundamentally challenging classical intuitions about containment and representing a paradigm shift from traditional set-theoretic measures.

\begin{remark}[Holographic Paradigm]
The axiom states that for holographic structures, we can have $I(M_1) > I(M_2)$ even when $M_1 \subsetneq M_2$ (proper subset). This is exemplified by:
\begin{itemize}
\item $I(P_{\text{ess}}) > I(\mathbb{N})$: The fractal prime set has strictly greater informational cardinality than the natural numbers, despite being a proper subset in a structural sense.
\item $I(P_{\text{ess}}) > I(C_{1/3})$: The fractal prime set has strictly greater informational cardinality than the Cantor set $C_{1/3}$, which has the cardinality of the continuum $\mathfrak{c}$.
\end{itemize}
\end{remark}

The axiom requires a notion of ``information content'' that is anti-monotone with respect to inclusion, reflecting how holographic encoding allows proper subsets to capture strictly more essential information than larger structures. This principle suggests that mathematical richness is not determined by size alone, but by the depth of arithmetic and geometric structure encoded within.

\subsection{Axiom (A6): Continuity}

Small changes in construction parameters (retention probabilities, scaling factors) should lead to continuous changes in $\iota$. This is related to the stability of L-function values under perturbations.

\subsection{Axiom (A7): Default Value}

Sets without known L-function connections (like $C_{1/3}$) have $\iota = 0$ by default. This makes informational cardinality computable for a wide class of objects while reserving non-zero values for arithmetically significant structures. Additionally, the trivial zeros of the Riemann zeta function (at negative even integers $s = -2, -4, -6, \ldots$) have informational cardinality $I(\text{trivial zeros}) = (0, 0, 0)$, as they lack the deep arithmetic significance of non-trivial zeros.

\section{Alternative Constructions}

\subsection{Primes Modulo 6}

Primes greater than 3 are congruent to 1 or 5 modulo 6. A fractal construction retaining positions 1 and 5 out of 6 gives:
\[
\dim_H = \frac{\log 2}{\log 6} \approx 0.387.
\]

The information measure could be defined as $\iota = -\zeta(s_0)$ where $s_0$ satisfies $\dim_H = s_0$, giving $s_0 \approx 0.387$.

\subsection{Primes Modulo 8}

Primes greater than 2 are odd, and odd primes modulo 8 are in $\{1,3,5,7\}$. Retaining 4 of 8 positions:
\[
\dim_H = \frac{\log 4}{\log 8} = \frac{2}{3}.
\]

This gives $\iota = -\zeta(2/3)$.

\subsection{Weighted Constructions}

Assign different probabilities to different residue classes based on empirical prime densities (incorporating Chebyshev bias). This creates non-uniform fractals with potentially different information measures.

\subsection{Higher-Dimensional Fractals}

Embed primes in $\mathbb{R}^2$ or higher dimensions using multiple moduli simultaneously (e.g., primes modulo 4 and modulo 3). This creates multi-dimensional fractal structures with richer geometry.

\section{Comparative Tables}

\begin{table}[h]
\centering
\begin{tabular}{|l|c|c|c|c|}
\hline
\textbf{Set} & $\alpha$ & $\delta$ & $\iota$ & $I(M)$ \\
\hline
$P_{\text{ess}}$ & 1 & $1/2$ & $-\zeta(1/2) \approx 1.46$ & $(1, 0.5, 1.46)$ \\
$C_{1/3}$ & 1 & $1/3$ & 0 & $(1, 0.333, 0)$ \\
$Z_F$ & 1 & $1/2$ & $\zeta(1/2) \approx -1.46$ & $(1, 0.5, -1.46)$ \\
$[0,1]$ & 1 & 1 & 0 & $(1, 1, 0)$ \\
Cantor set & 1 & $\log 2/\log 3 \approx 0.631$ & 0 & $(1, 0.631, 0)$ \\
\hline
\end{tabular}
\caption{Informational cardinalities of various sets}
\end{table}

\begin{table}[h]
\centering
\begin{tabular}{|l|c|c|}
\hline
\textbf{Property} & $P_{\text{ess}}$ & $Z_F$ \\
\hline
Hausdorff dimension & $1/2$ & $1/2$ \\
Cardinality & $\mathfrak{c}$ & $\mathfrak{c}$ \\
Lebesgue measure & 0 & 0 \\
Self-similar & Yes (deterministic) & Yes (data-dependent) \\
Information measure & $-\zeta(1/2) > 0$ & $\zeta(1/2) < 0$ \\
Arithmetic origin & Primes mod 4 & Zeta zeros \\
\hline
\end{tabular}
\caption{Comparison of $P_{\text{ess}}$ and $Z_F$}
\end{table}

\section{Computational Aspects}

\subsection{Computing $P_{\text{ess}}$}

The set $P_{\text{ess}}$ can be approximated to arbitrary precision by computing $I_n$ for large $n$. At stage $n$, there are $2^n$ intervals, each of length $4^{-n}$. The endpoints can be computed exactly using base-4 arithmetic.

\subsection{Computing $Z_F$}

Computing $Z_F$ requires:
\begin{enumerate}
\item Numerical computation of zeta zeros $\gamma_n$ (available in databases like LMFDB).
\item Extraction of fractional parts $t_n = \{\gamma_n/(2\pi)\}$.
\item Conversion to base-4 digits $a_n = \lfloor 4t_n \rfloor$.
\item Iterative construction of $J_n$.
\end{enumerate}

The first $10^{10}$ zeros are known, allowing computation of $J_n$ for $n \le 10^{10}$.

\subsection{Numerical Verification of Information Conservation}

To test the conjecture $\iota(P_{\text{ess}}) + \iota(Z_F) = 0$, we need:
\begin{enumerate}
\item Precise computation of $\zeta(1/2)$ (done: $\zeta(1/2) \approx -1.46035$).
\item Verification that $Z_F$ has the expected statistical properties.
\item Numerical estimation of $\iota(Z_F)$ via correlation with prime distribution.
\end{enumerate}

This remains an open computational challenge.

\end{document}